\documentclass[preprint,12pt]{article}   

\RequirePackage[OT1]{fontenc}
\RequirePackage{amsthm,amsmath}
\usepackage{amsthm,amsfonts,amssymb,amscd}
\usepackage{bbm}
\usepackage{stmaryrd}
\RequirePackage[numbers]{natbib}
\RequirePackage[colorlinks,citecolor=blue,urlcolor=blue]{hyperref}

\usepackage{bbm}
\usepackage{xcolor}
\usepackage{graphicx}
\usepackage{caption}
\usepackage{subcaption}
\usepackage{bm}
\usepackage{ulem}

\newcommand{\Esp}{\mathbb{E}}
\newcommand{\E}{\mathbb{E}}
\newcommand{\Var}{\mathrm{var}}

\newcommand{\sgn}{\mathrm{sgn}}
\newcommand{\Hilb}{\mathcal{H}}
\newcommand{\R}{\mathbb{R}}
\newcommand{\N}{\mathbb{N}}
\newcommand{\T}{\mathcal{T}}

\newcommand{\Dt}{D_1^{\textrm{tot}}}

\newcommand{\tot}{\text{tot}}
\newcommand{\multi}{{\underline{\ell}}}
\newcommand{\Poinc}{C_\textrm{P}}

\newtheorem{cor}{Corollary}
\newtheorem{prop}{Proposition}
\newtheorem{defi}{Definition}
\newtheorem{rem}{Remark}
\newtheorem{example}{Example}

\usepackage{authblk}

\title{Sensitivity Analysis and Generalized Chaos Expansions. Lower Bounds for Sobol indices.}

\author[1]{O. Roustant}
\author[2]{F. Gamboa}
\author[3,2]{B. Iooss}

\affil[1]{{\small Mines Saint-\'{E}tienne, Univ. Clermont Auvergne, CNRS, UMR 6158 LIMOS, F--42023 Saint-\'{E}tienne, France}}
\affil[2]{{\small Institut de Math\'ematiques de Toulouse, Universit\'e Paul Sabatier, 31062 Toulouse Cedex 9, France}}
\affil[3]{{\small Electricit\'e de France R\&D, 6 quai Watier, Chatou, F-78401, France}}
\date{} 

\begin{document}

\maketitle

\begin{abstract}
The so-called polynomial chaos expansion is widely used in computer experiments. For example, it is a powerful tool to estimate Sobol' sensitivity indices. In this paper, we consider generalized chaos expansions built on general tensor Hilbert basis.
In this frame, we revisit the computation of the Sobol' indices and give general lower bounds for these indices. The case of the eigenfunctions system associated with a Poincar\'e differential operator leads to lower bounds involving the derivatives of the analyzed function and provides an efficient tool for variable screening. These lower bounds are put in action both on toy and real life models demonstrating their accuracy.    
\end{abstract}

\tableofcontents


\section{Introduction}

Computer models simulating physical phenomena and industrial systems are commonly used in engineering and safety studies. They often take as inputs a high number of numerical and physical variables. 
For the development and the analysis of such computer models, the global sensitivity analysis methodology is an invaluable tool that allows to rank the relative importance of each input of the system \cite{ioosal17}, \cite{Iooss_Lemaitre_review}.
Referring to a probabilistic modeling of the model input variables, it accounts for the whole input range of variation, and tries to explain output uncertainties on the basis of input uncertainties. 
Thanks to the so-called functional ANOVA (analysis of variance) decomposition \cite{Antoniadis_1984}, the Sobol' indices give, for a square integrable non-linear model and stochastically independent input variables, the parts of the output variance due to each input and to each interaction between inputs \cite{Sobol_1993}, \cite{Hoeffding_1948}. In addition, the total Sobol' index provides the overall contribution of each input \cite{homsal96}, including interactions with other inputs.  More generally, we recall that a Sobol' index associated to a subset of variables $I$
is the ratio of the ANOVA index (that is the $L^2$ norm of the contribution associated to $I$ in the ANOVA decomposition), and the variance of the output (see Section \ref{sec:background} for the precise definition). 

Many methods exist to accurately compute or statistically estimate the first-order Sobol' indices. For a general overview on these methods, we refer to \cite{Iooss_Lemaitre_review} and references therein.
One of the most popular and powerful method is polynomial chaos (PC) expansion \cite{ghaspa91}, \cite{sud08a}.
It consists in approximating the response onto the specific basis made by the orthonormal polynomials built on the input distributions. 
Its strength stands on the fact that, once the expansion is computed, the Parseval formula gives directly all the ANOVA indices (in particular the total Sobol' indices) \cite{sud08a,cremar09}. 
Of course in practice the PC expansion is truncated.
An obvious but important fact is that this truncated PC expansion provides a lower bound for the true ANOVA index. 
In this paper, we consider general tensor Hilbert basis called generalized chaos (GC). 
Further, we use the previous trick to produce general lower bounds (see Section \ref{sec:GC_expansions}). 
Then, a smart choice of the GC produces new interesting lower bounds involving the derivatives of the function of interest (see Section \ref{sec:DO_expansions}). More precisely, this special Hilbert basis is obtained by diagonalizing the Poincar\'e differential operators (PDO), associated with the input distributions (this operator is related to Poincar\'e inequality, see \cite{BGL_book} or \cite{Bonnefont_Joulin_Ma_weighted_Poincare}). Notice that other special GC expansions based on the diagonalization of reproducing kernels has been recently studied and used for global sensitivity purposes in  \cite{pro19}.

In general, the estimation of the total Sobol' indices (and other ANOVA indices) suffers from the curse of dimensionality (number of inputs) and can be too costly in terms of number of model evaluations \cite{pritar17}.
Low-cost computations of upper and lower bounds for total Sobol' indices are then very useful.
DGSM (Derivative-based Global Sensitivity Measures, see \cite{sobkuc09}), computed from some integral of the squared derivatives of the model output, may give such economical upper and lower bounds \cite{kucrod09,Kucherenko_Iooss_book}. 
Indeed, in many physical models the so-called adjoint method allows at weak extra cost the evaluation of the derivatives of the model (see for example the recent review \cite{Allaire_2015}).
Concerning the upper bounds, optimal and general (for any distribution type of the input) results are obtained in \cite{Roustant_Barthe_Iooss_2017}.
For lower bounds, only special cases (uniform, Normal and Gamma) have been investigated in \cite{Sudret_Mai_2015, Kucherenko_Song_2016} (see \cite{Kucherenko_Iooss_book} for a review).
The bounds given in \cite{Kucherenko_Song_2016} are quite rough as they are smaller than the first-order Sobol' indices.
In our work, we follow the tracks opened by \cite{Sudret_Mai_2015} using PC expansions, but for both much more general distributions and expansions. Indeed, for a wide class of input distributions the PDO generalized chaos expansion leads naturally to quantities built on the derivatives.

Notice that the diagonalization of PDO used here, 
leads to orthogonal polynomial only for the Gaussian distribution (see \cite{BGL_book} and \cite{bakry2003characterization}). Indeed, the PDO considered in this paper 
only involves the integration with respect to the input distribution of the squared derivatives 
(and not a reweighted input distribution).   
Apart from this particular probability distribution,
orthogonal polynomials cannot be interpreted, in general, as eigenfunctions of a PDO. 
Consequently, in general the Hilbert basis built by diagonalizing a PDO is not a
polynomial basis.
For example, for the uniform distribution, it is the Fourier basis.

The paper is structured as follows.
Section~\ref{sec:background} recalls the required mathematical tools for global sensitivity analysis (ANOVA decomposition and DGSM).
Section~\ref{sec:GC_expansions} rephrases the ANOVA decomposition with Hilbert spaces, and introduces the generalized chaos expansion.
Section~\ref{sec:DO_expansions} then focuses on PDO expansions, and their link to PC expansions.
Section~\ref{sec:weight_free_DGSM} gives an alternative proposition of orthonormal functions which lead to weight-free DGSM. Section~\ref{sec:examples} gives analytical examples.
Section~\ref{sec:applications} illustrates on real life applications.
Section~\ref{sec:conclusion} gives some perspectives for future works.


\section{Background on sensitivity analysis}  \label{sec:background}
To begin with, let
$X = (X_1, \dots, X_d)$ denotes the vector of independent input variables  
with distribution $\mu = \mu_1 \otimes \dots \otimes \mu_d$.
Here the $\mu_i$'s are continuous probability measures on $\R$.
Let further $h$ be a multivariate function of interest $h: \Delta \subseteq \mathbb{R}^d \rightarrow \mathbb{R}$. 
We assume that $h(X) \in \mathcal{H} := L^2(\mu)$.

One of the main tool in global sensitivity analysis is the Sobol'-Hoeffding decomposition of $h$,
(see \cite{Hoeffding_1948, Efron_Stein_1981, Antoniadis_1984, Sobol_1993}).
It provides a unique expansion of $h$ as
$$h(X) = h_0 + \sum_{i=1}^d h_i(X_i) + \sum_{1 \leq i<j \leq d} h_{i,j}(X_i, X_j) 
+ \dots + h_{1, \dots, d}(X_1, \dots, X_d) \nonumber$$
with $\E [h_I(X_I)\vert X_J]=0$ for all $I \subseteq \{1, \dots, d\}$ and all $J \subsetneq I$ 
(with the notation $\mbox{$X_I:=(X_i:\;i\in I)$})$. 
Furthermore, $h_0 = \E[h(X)]$ and
$$h_I(X_I) = \E[h(X) \vert X_I] - \sum_{J \subsetneq I} h_J(X_J) 
= \sum_{J \subseteq I} (-1)^{\vert I \vert - \vert J \vert} \E[h(X) \vert X_J].$$ 
Notice that the condition
$$\E [h_I(X_I)\vert X_J]=0 \quad \textrm{for all} \quad J \subsetneq I,$$
warrants both the uniqueness of the decomposition and the orthogonality of $h_I(X_I)$
to any square integrable random variables depending only on $X_J$ with $J \cap I \subsetneq I$. 

This last property leads to the so-called ANOVA decomposition for the variance of $h(X)$

\begin{equation}
D := \Var(h(X)) = \sum_{I \subseteq \{1, \dots, d \}} \Var (h_I(X_I)).
\label{eq:DD} 
\end{equation}

Notice further that the Sobol'-Hoeffding decomposition is a particular case of the multivariate decomposition built on 
a  finite family of commuting projectors $P_1, \dots, P_d$ and obtained by expanding the following product (see \cite{Kuo_et_al_2010}),
\begin{eqnarray*}
I_d &=& (P_1 + (I_d - P_1)) \dots (P_d + (I_d - P_d))\\
 &=& \sum_{I \subseteq \{1, \dots, d \}} \underset{{\Pi_I}}{\underbrace{\prod_{j \notin I} P_j \prod_{k \in I} (I - P_k)}}.
\end{eqnarray*}
Obviously, $\Pi_I$ is also a projector.
In the Sobol'-Hoeffding decomposition the projection $P_j h$ is  
$\int h(x) d\mu_j(x_j)$.

In sensitivity analysis, one classically considers the 
Sobol' indices. These indices are defined, for $I\subseteq \{1, \dots, d\}$, as $S_I = D_I / D$ where $D_I:= \Var (h_I(X_I))$. From (\ref{eq:DD}) one directly obtains
$$D = \sum_I D_I, \qquad \qquad 1 = \sum_I S_I.$$ 
Another interesting index is the total Sobol' one that includes all the contributions on the total variance of a variable group. In this paper, the total index associated to one variable is the object under study. For $\mbox{$I\subseteq \{1, \dots, d\}$}$, the total Sobol' index associated to $I$ is  defined as  $S_I^\tot :=  \frac{D_I^\tot}{D}$ with 
$$\mbox{$D_I^\tot :=  \sum_{J \supseteq \{ I \}}D_I$}.$$

To end this section, we recall the other popular global sensitivity index that will appear in our bounds. This is the so-called  Derivative Global Sensitivity Measure (DGSM) introduced and studied in \cite{Sobol_Gresham_1995} and \cite{kucrod09}. It is defined, for $I\subseteq \{1, \dots, d\}$, under smoothness and integrability assumptions on $h$ as 
$$\nu_I = \int \left(\frac{\partial^{\vert I \vert} h(x)}{\partial x_I}\right)^{2}\mu(dx).$$
 

\section{Generalized chaos expansions}   \label{sec:GC_expansions}
In order to present the generalized chaos expansions, 
it is convenient to first rephrase the classical functional ANOVA decomposition 
presented in the previous section as a Hilbert space decomposition.
The next proposition is devoted to this task.
In particular, we emphasize that the operator giving one ANOVA term is an orthogonal projection.
Then, we discuss the construction of Hilbert basis tailored to ANOVA decomposition. 
Part of the material is inspired from \cite{tissotThese} and \cite{Antoniadis_1984}.


\begin{prop}[Hilbert space decomposition for ANOVA]
For all subset $I$ of $\{1, \dots, d\}$, the map $\Pi_I: h \in \Hilb \mapsto h_I$ is an orthogonal projection.
The image spaces ${\Hilb_I = \Pi_I(\Hilb) = \{h \in \Hilb, \, h=h_I \}}$, called \emph{ANOVA spaces}, 
are Hilbert spaces that form an orthogonal decomposition of $\Hilb$:
\begin{equation} \label{eq:HilbertDecomposition}
\Hilb = \underset{I \subseteq \{1, \dots, d \}}{\overset{\perp}{\oplus}} \Hilb_I
\end{equation}
\end{prop}

\begin{proof}
First, $\Pi_I$ is a projector since applying twice the ANOVA decomposition leaves it unchanged.
Now, let $g, h \in \Hilb$. We have:
$$\langle \Pi_I g, h \rangle = \E(g_I(X_I) h(X)) 
= \sum_{J \subseteq \{1, \dots, d\}} \E(g_I(X_I) h_J(X_J))$$
where we wrote the ANOVA decomposition of $h$. 
Now, if $J \neq I$, then $I \cap J \subsetneq I$ or $I \cap J \subsetneq J$, thus $\E(g_I(X_I) h_J(X_J)) = 0$
by the uniqueness property of ANOVA decomposition.
Hence, 
$$ \langle \Pi_I g, h \rangle = \E(g_I(X_I) h_I(X_I)) = \langle g, \Pi_I h \rangle, $$
which proves that the projector $\Pi_I$ is self-adjoint, and thus orthogonal.

Consequently, $\Pi_I$ is continuous and $\Hilb_I$ is a Hilbert space as a closed subspace of $\Hilb$.
The direct sum (\ref{eq:HilbertDecomposition}) results from the existence and uniqueness of ANOVA decomposition.
As shown above, the uniqueness property implies that $\Hilb_I \perp \Hilb_J$ if $I \neq J$.
\end{proof}

\begin{cor}[Hilbert space decomposition for total effects]
Let $I$ be a subset of $\{1, \dots, d\}$.
Then the map $\Pi_I^\tot: h \in \Hilb \mapsto h_I^\tot = \sum_{J \supseteq I} h_J$ is an orthogonal projection.
The image space ${\Hilb_I^\tot = \Pi_I^\tot(\Hilb) = \{h \in \Hilb, \, h=h_I^\tot \}}$ is the Hilbert space
\begin{equation} \label{eq:HilbertTot}
\Hilb_I^\tot = \underset{J \supseteq I}{\overset{\perp}{\oplus}} \Hilb_J.
\end{equation}
\end{cor}

\begin{proof}
Observe that $\Pi_I^\tot = \sum_{J \supseteq I} \Pi_J$. 
As the $\Pi_J$ are commuting orthogonal projections, $\Pi_I^\tot$ is an orthogonal projection.
The remainder is straightforward.
\end{proof}


We now exhibit Hilbert bases of $\Hilb$ that are adapted to the ANOVA decomposition, 
in the sense that each element belongs to one ANOVA space $\Hilb_I$. 
This provides Hilbert bases for all $\Hilb_I$ and $\Hilb_I^\tot$.

\begin{defi}[Generalized chaos] \label{def:chaosBases}
For $i = 1, \dots, d$, let $(e_{i,n})_{n \in \N}$ be a Hilbert basis of $L^2(\mu_i)$, with $e_{i,0} = 1$.
For a multi-index $\multi = (\ell_1, \dots, \ell_d) \in \N^d$, 
the \emph{generalized chaos of order $\multi$} is defined as the following $L^2(\mu)$ function:
$$ e_\multi(x) := \left( \underset{i=1, \dots, d} \otimes e_{i, \ell_i} \right) (x) 
=  e_{1, \ell_1}(x_1) \times \dots \times e_{d, \ell_d}(x_d).$$
\end{defi}

The so-called polynomial chaos introduced by \cite{wiener1938homogeneous}, 
built with the orthogonal polynomials associated to the Gaussian distribution (Hermite polynomials $(H_n)$),
is a special case of the previous definition (with $e_{i,n} = H_n$).
Similarly, this is also the case for the generalized polynomial chaos
corresponding to orthogonal polynomials associated to other probability distributions.
For history on polynomial chaos and generalized polynomial chaos, 
we refer to the introduction of \citep{ernst2012convergence}.
Other examples of generalized chaos in the context of sensitivity analysis are the Fourier bases, 
investigated in \cite{cukier1978nonlinear},
and the Haar systems originally used by Sobol' \cite{sobol1969multidimensional}.

\begin{prop} \label{prop:HilbertBasis}
$\;$
\begin{enumerate}
\item The whole set of generalized chaos $\T := (e_\multi)_{\multi \in \N^d}$ is a Hilbert basis of $\Hilb$, 
and each $e_\multi$ belongs to (exactly) one $\Hilb_I$, 
where $I$ is the set containing the indices of active variables:  
$I = \{i \in \{1, \dots, d\}: \, \ell_i \geq 1 \}$.
\item For all $I \subseteq \{1, \dots, d\}$, 
\begin{itemize}
\item The subset of basis functions that involve \emph{exactly} the variables in $I$,
$\T_I := \{e_\multi, \, 
\mbox{ with }  \ell_i \geq 1 \mbox{ if } i \in I 
\mbox{ and }  \ell_i = 0 \mbox{ if } i \notin I
\}$ 
is a Hilbert basis of $\Hilb_I$.
\item The subset of basis functions that involve \emph{at least} the variables in $I$,
$\T_I^\tot := \{e_\multi, \, 
\mbox{ with }  \ell_i \geq 1 \mbox{ if } i \in I 
\}$ 
is a Hilbert basis of $\Hilb_I^\tot$. 
\end{itemize}
\end{enumerate}
\end{prop}

Notice that in the definition of $\T_I$ and $\T_I^\tot$, 
the index $n_i$ 
is non zero, which means that $x_i$ is active.

\begin{proof}
The fact that $\T$ is a Hilbert basis of $\Hilb$ is well known.
Let us see that $e_\multi$ belongs to $\Hilb_I$, with $I = \{i \in \{1, \dots, d\}: \, \ell_i \geq 1 \}$.
For that, we need to check that the ANOVA decomposition of $e_\multi$ 
consists of only one non-zero term corresponding to the subset $I$ and equal to $e_\multi$.
As $e_\multi$ is a function of $x_I$, 
it remains to check the non-overlapping condition. 
Let $J$ be a strict subset of $I$ (possibly empty). 
Then, 
$$ \E \left[ \prod_{i \in I} e_{i, \ell_i}(X_I) \vert X_J \right] =  
\prod_{j \in J} e_{j, \ell_j}(X_j) \prod_{i \in I \setminus J} \E \left[ e_{i, \ell_i}(X_I) \right] $$
Let us choose $i \in I \setminus J$. 
Then, $\ell_i \geq 1$, implying that $\E \left[ e_{i, \ell_i}(X_I) \right] = 0$
(as $e_{i, \ell_i}$ is orthogonal to $e_{i, 0} = 1$). 
Finally $e_\multi$ belongs to $\Hilb_I$.
Now let us fix a subset $I$ of $\{1, \dots, d\}$, and consider for instance $\T_I$
(the proof is similar for $\T_I^\tot$).
Clearly, as a subset of $\T$, the set $\T_I$ is a collection of orthonormal functions.
Furthermore, by the proof above, each $e_\multi$ of $\T_I$ belongs to $\Hilb_I$. 
To see that $\T_I$ is dense in $\Hilb_I$, let us choose $h \in \Hilb_I$.
Since $\T$ is a Hilbert basis of $\Hilb$, then $h$ can be written as
$$ h = \sum_{\multi \in \N^d} c_\multi e_\multi 
= \sum_{e_\multi \in \T_I} c_\multi e_\multi + \sum_{e_\multi \notin \T_I} c_\multi e_\multi$$
where $(c_\multi)_{\multi \in \N^d}$ is a squared integrable sequence of real numbers.
Recall that each $e_\multi$ belongs to $\Hilb_J$,
with $J = \{i \in \{1, \dots, d\} \, s.t. \, \ell_i \geq 1 \}$.
Thus, if $e_\multi \notin \T_I$, then $J \neq I$.
Hence, $e_\multi \in \Hilb_I^\perp$ (as $\Hilb_J \perp \Hilb_I$). 
Since $h \in \Hilb_I$, it implies that 
$ \sum_{e_\multi \notin \T_I} c_\multi e_\multi = 0$.
\end{proof}

The previous results imply that the variance $D_I$ (resp. $D_I^\tot$) of the output 
explained by a set $I$ (resp. supersets of $I$) of input variables,
is equal to the squared norm of the orthogonal projection onto $\Hilb_I$ (resp. $\Hilb_I^\tot$).
Hence, lower bounds can be obtained by projecting onto smaller subspaces.

\begin{cor} \label{prop:LB_general}
Let $I$ be a subset of $\{1, \dots, d\}$ and let $h \in \Hilb$. Then:
\begin{itemize}
\item For all subset $G$ of $\Hilb_I$, $D_I = \Vert \Pi_I (h) \Vert ^2 \geq \Vert \Pi_G(h) \Vert ^2$, 
with equality iff $h$ has the form $h = g + f$ with $g \in G$ and $f \in \Hilb_I^\perp$
\item For all subset $G$ of $\Hilb_I^\tot$, $D_I^\tot = \Vert \Pi_I^\tot (h) \Vert ^2 \geq \Vert \Pi_G(h) \Vert ^2$,
with equality iff $h$ has the form $h = g + f$ with $g \in G$ and $f \in (\Hilb_I^\tot)^\perp$
\end{itemize}
\end{cor}

In practice, the subset $G$ on which to project may be finite dimensional.
For instance, it can be chosen by picking a finite number of orthonormal functions 
from the Hilbert basis obtained in Proposition~\ref{prop:HilbertBasis}.
We illustrate this on the common case where $I$ correspond to a single variable.
Without loss of generality, we assume that $I = \{1\}$. 

\begin{cor} \label{prop:LB_general_1var}
Let $\phi_1, \dots, \phi_N$ be orthonormal functions in $\Hilb_1^\tot$. Then:
$$ D_1^\tot(h) \geq \sum_{n=1}^N \left( \int h(x) \phi_n(x) \mu(dx) \right)^2$$
with equality iff $h$ has the form
$h(x) = \sum_{n=1}^N \alpha_n \phi_n(x) + g(x_2, \dots, x_N)$,
where $g \in L^2(\underset{i=2, \dots, d} \otimes \mu_i)$.
Furthermore, if all the $\phi_j$'s belong to $\Hilb_1$, then the lower bound holds for $D_1$.
\end{cor}

\begin{proof}
This is a direct application of Corollary~\ref{prop:LB_general} 
with $G = \textrm{span} \{ \phi_1, \dots, \phi_m \}$.
The equality case is obtained by remarking that 
$(\Hilb_1^\tot)^\perp$
is formed by functions of $\Hilb$ that do not involve $x_1$:
$(\Hilb_1^\tot)^\perp = \underset{J \subseteq \{2, \dots, d\}} \oplus \Hilb_J$. 
\end{proof} 


\section{Poincar\'e differential operator expansions} \label{sec:DO_expansions}
Generalized chaos expansions are defined from $d$ Hilbert bases 
associated to probability measures on the real line $\mu_i$ ($i = 1, \dots, d)$.
Here, each $\mu_i$ is assumed to be absolutely continuous with respect to the Lebesgue measure. 
In this section, we exhibit a class of Hilbert basis which is well tailored 
to perform sensitivity analysis based on derivatives.
They consist of eigenfunctions of an elliptic differential operator (DO).
More precisely, we choose the DO associated to a 1-dimensional Poincar\'e inequality (assuming it holds)
\begin{equation} \label{eq:Poincare}
\Var_{\mu_1}(h) \leq C \int_{\R} h'(x)^2 \mu_1(dx),
\end{equation}
as it was successfully used to obtain accurate bounds for DGSM \cite{Roustant_Barthe_Iooss_2017}.

Before defining the so-called PDO expansions, 
we first recall the spectral theorem related to Poincar\'e inequalities.
In what follows, for any positive integer $\ell$, 
we denote by $H^\ell(\mu_1)$ the Sobolev space of order $\ell$:
\begin{equation} \label{def:SobolevSpace}
H^\ell(\mu_1) := \{h \in L^2(\mu_1) \text{ such that for all } k \leq \ell, h^{(k)} \in L^2(\mu_1) \} 
\end{equation}

\begin{prop}[Spectral theorem for Poincar\'e inequalities, \cite{BGL_book, Roustant_Barthe_Iooss_2017}] \label{prop:spectralTheorem}
Let $\mu_1(dt) = \rho(t) dt$ be a continuous measure on a bounded interval 
$I = (a,b)$ of $\R$, where $\rho(t) = e^{-V(t)}$. 
Assume that $V$ is continuous and piecewise $C^1$ on $\bar{I} = [a, b]$.
Then consider the differential operator 
\begin{equation} \label{eq:PoincareOperator}
Lh = h'' - V'h'
\end{equation}
defined on $\Hilb' = \{ h \in H^2(\mu_1) \text{ s.t. } h'(a) = h'(b) = 0 \}$.
Then $L$ admits a spectral decomposition. That is, 
there exists an increasing sequence $(\lambda_n)_{n \geq 0 }$ of non-negative values that tends to infinity,
and a set of orthonormal functions $e_n$ which form a Hilbert basis of $L^2(\mu_1)$ such that $Le_n = - \lambda_n e_n$.
Furthermore, all the eigenvalues $\lambda_n$ are simple. 
The first eigenvalue is $\lambda_0 = 0$, and the corresponding eigenspace consists of constant functions
(we can choose $e_0 = 1$).
The first positive eigenvalue $\lambda_1$ is called spectral gap, 
and equal to the inverse of the Poincar\'e constant $\Poinc(\mu_1)$, 
i.e. the smallest constant satisfying Inequality (\ref{eq:Poincare}).\\
\end{prop}

\begin{rem} \label{rem:spectralConditions}
The assumptions of Proposition~\ref{prop:spectralTheorem} guarantee that $L$ admits a spectral decomposition,
and correspond to a continuous probability distribution defined on a compact support,
whose density is continuous and does not vanish.
However, the spectral decomposition can exist for more general cases.
For instance, it exists for the Normal distribution on $\R$:
the corresponding eigenfunctions consist of Hermite polynomials and eigenvalues to non-negative integers.
On the other hand, the spectral decomposition does not exist for the Laplace (double-exponential) distribution 
on the whole $\R$.\\
\end{rem}

The key property in our context is given by the equation
\begin{equation} \label{eq:intPart}
\langle h', e'_n \rangle =  \lambda_n  \langle h, e_n \rangle,
\end{equation}
corresponding to the weak formulation of the spectral problem $Le_n = -\lambda_n e_n$ associated to the Poincar\'e inequality, 
and holding for all $n \geq 0$, and all $h \in H^1(\mu_1)$.
It implies that geometric quantities involved in PDO expansions 
can be rewritten with derivatives.
In particular, for a centered function $h$, 
we have:
$$ \Vert h \Vert ^2 = \sum_{n=1}^\infty \langle h, e_n \rangle^2 
= \sum_{n=1}^\infty \frac{1}{\lambda_n^2} \langle h', e'_n \rangle^2.$$

Let us come back to the $d$-dimensional situation, where $\mu = \underset{i=1, \dots, d}\otimes \mu_i$.
For each measure $\mu_i$, we make the assumptions of Proposition~\ref{prop:spectralTheorem}
(see also Remark~\ref{rem:spectralConditions} for alternative conditions).
We denote by $L_i$ the corresponding operator and 
$\lambda_{i,n}, e_{i,n}$ ($n \geq 0$) its eigenvalues and eigenfunctions.
We define $H^1(\mu)$ similarly to $H^1(\mu_1)$ (Equation~\ref{def:SobolevSpace}).
We can now define the PDO expansion and then state the main result.\\

\begin{defi}[PDO expansions] \label{def:DOexpansion}
We call \emph{Poincar\'e differential operator (PDO) expansion} the generalized chaos expansion corresponding 
to the Hilbert bases formed by the eigenfunctions of $L_1, \dots, L_d$. 
\end{defi}

\begin{prop}[Poincar\'e-based lower bounds] \label{prop:PoincareLB}
For all $h$ in $H^1(\mu)$, we have
\begin{eqnarray} \label{eq:diffOpPythagore}
D_1^\tot(h) 
&=& \sum_{\ell_1 \geq 1, \ell_2, \dots, \ell_d} \langle h, e_{1,\ell_1} \dots e_{d, \ell_d} \rangle^2 \label{eq:diffOpLB1}  \\
&=& \sum_{\ell_1 \geq 1, \ell_2, \dots, \ell_d} \frac{1}{\lambda_{1,\ell_1}^2} \langle \frac{\partial h}{\partial x_1}, e'_{1,\ell_1} e_{2, \ell_2} \dots e_{d,\ell_d} \rangle^2.  \label{eq:diffOpLB2}
\end{eqnarray}
In particular, limiting ourselves to the first eigenfunction in all dimensions, 
and to first and second order tensors involving $x_1$, we obtain the lower bound
\begin{eqnarray} 
D_1^\tot(h)  
&\geq & \langle h, e_{1,1}\rangle^2 + \sum_{i=2}^d  \langle h, e_{1,1} e_{i, 1}\rangle^2 \label{ineq:diffOpLB1} \\ 
& = & \Poinc(\mu_1)^2 \left( \langle \frac{\partial h}{\partial x_1}, e'_{1,1} \rangle^2 
+ \sum_{i=2}^d  \langle \frac{\partial h}{\partial x_1}, e'_{1,1} e_{i, 1} \rangle^2 \right). \label{ineq:diffOpLB2}
\end{eqnarray}
\end{prop}

\begin{proof}
By Proposition~\ref{prop:HilbertBasis}, the subset of ($e_\multi$) corresponding to $\ell_1 \geq 1$ is a Hilbert basis of $\Hilb_1^\tot$.
This gives (\ref{eq:diffOpLB1}).
Now, for  $\ell_1 \geq 1$:
$$ \langle h, e_{1,\ell_1} \dots e_{d,\ell_d} \rangle = 
\frac{1}{\lambda_{1, \ell_1}} \langle \frac{\partial h}{\partial x_1}, e'_{1,\ell_1} e_{2, \ell_2} \dots e_{d,\ell_d} \rangle $$
This is obtained by applying Eq.~(\ref{eq:intPart}) to $x_1 \mapsto h(x)$ 
and integrating with respect to $x_2, \dots, x_d$:
\begin{eqnarray*}
\langle h, e_{1,\ell_1} \dots e_{d,\ell_d} \rangle &=&
\int \langle h(\bullet, x_2, \dots, x_d), e_{1,\ell_1} \rangle_{L^2(\mu_1)} 
\prod_{i=2}^d e_{i,\ell_i} \mu_i(dx_i) \\
&=& \frac{1}{\lambda_{1, \ell_1}} \int \langle \frac{\partial h(\bullet, x_2, \dots, x_d)}{\partial x_1}, e'_{1,\ell_1} \rangle_{L^2(\mu_1)} 
\prod_{i=2}^d e_{i,\ell_i} \mu_i(dx_i)\\
&=& \frac{1}{\lambda_{1, \ell_1}} \langle \frac{\partial h}{\partial x_1}, e'_{1,\ell_1} e_{2, \ell_2} \dots e_{d,\ell_d} \rangle
\end{eqnarray*}
This gives (\ref{eq:diffOpLB2}). 
The remainder is straightforward, knowing that ${\Poinc(\mu_1) = 1 / \lambda_{1,1}}$.
\end{proof}

\paragraph{Case of uniform distributions: Fourier expansion.}
Let us assume that $\mu_1$ is uniform on $[-1/2, 1/2]$. 
Then, the differential operator $L$ is the usual Laplacian, and its eigenfunctions correspond to Fourier basis.
More precisely, using the Neumann boundary conditions $h'(a) = h'(b) = 0$, 
one can check that the eigenvalues are 
$ \lambda_\ell = \ell^2 \pi^2 $, $(\ell=0, 1, \dots)$, 
and a set of orthonormal eigenfunctions is given by 
$e_0 = 1$ and 
$$e_\ell(x_1) = \sqrt{2}{\cos ( \pi \ell (x_1 + 1/2))}$$
for $\ell > 0$.
Denote by $\vert \underline{\ell} \vert_0$ the number of non-zero coefficients of the multi-index $\underline{\ell} = (\ell_1, \dots, \ell_d)$. 
When the other $\mu_i$'s are also uniform on $[-1/2, 1/2]$,
 we obtain a multivariate Parseval formula for $D_1^\tot$:
\begin{eqnarray*} 
&& D_1^\tot(h)
= \sum_{\ell_1 \geq 1, \ell_2, \dots, \ell_d} 2^{\vert \underline{\ell} \vert_0}
\langle h, \prod_{i=1}^d \cos ( \pi \ell_i (x_i + 1/2)) \rangle^2 \\
&=& \sum_{\ell_1 \geq 1, \ell_2, \dots, \ell_d} 2^{\vert \underline{\ell} \vert_0}
\frac{1}{ \pi^2 \ell_1^2} \langle 
\frac{\partial h}{\partial x_1}, \sin ( \pi \ell_1 (x_1 + 1/2)) \prod_{i=2}^d \cos ( \pi \ell_i (x_i + 1/2)) \rangle^2
\end{eqnarray*}
Limiting for instance the sum to first terms, we obtain the lower bounds
\begin{eqnarray} 
D_1^\tot(h)  
&\geq & 2 \langle h, \sin ( \pi x_1) \rangle^2 
+ 4 \sum_{i=2}^d  \langle h, \sin (\pi x_1) \sin (\pi x_i) \rangle^2 \\ \label{ineq:DOFourier}
&\geq & \frac{2}{\pi^2} \left( \langle \frac{\partial h}{\partial x_1}, \cos(\pi x_1) \rangle^2 
+ 2 \sum_{i=2}^d  \langle \frac{\partial h}{\partial x_1}, \cos(\pi x_1) \sin(\pi x_i) \rangle^2 \right) \label{ineq:DOFourierDer}
\end{eqnarray}

\paragraph{Extension of PDO expansions to weighted Poincar\'e inequalities.} 
PDO expansions correspond to diffusion operators associated to Poincar\'e inequalities.
They can be extended to weighted Poincar\'e inequalities
\begin{equation} \label{eq:weightedPoincare}
\Var_{\mu_1}(h) \leq C \int_{\R} h'(x)^2 w(x) \mu_1(dx),
\end{equation}
defined for some suitable positive weight $w$.
Such inequalities have recently been used in sensitivity analysis \cite{Song_Weighted_Poincare}.
They are also useful when a probability distribution does not admit a Poincar\'e inequality
such as the Cauchy distribution \cite{Bonnefont_Joulin_Ma_weighted_Poincare}.
The weighted Poincar\'e inequality (\ref{eq:weightedPoincare}) corresponds to the differential operator
\begin{equation} \label{eq:weightedPoincareOperator}
L h = w h '' + (w' - w V') h'.
\end{equation}
Similarly to (\ref{eq:intPart}), rewriting geometrical quantities with derivatives 
can be done with the formula:
\begin{equation} \label{eq:intPartWeight}
\langle h', e'_n \rangle_w =  \lambda_n  \langle h, e_n \rangle,
\end{equation}
where $\langle ., . \rangle_w$ is the weighted dot product  
$\langle f, g \rangle_w := \int f(x)g(x)w(x)\mu(dx)$.
Proposition~\ref{prop:PoincareLB} can be adapted accordingly.

\paragraph{When PDO expansions coincide with PC expansions.}
There are exactly three cases where PDO expansions coincide with PC expansions, 
even when considering their extension to weighted Poincar\'e inequalities.
Indeed, it can be shown that orthogonal polynomials are eigenfunctions of diffusion operators 
only for the Normal, Gamma and Beta distributions, corresponding respectively to 
Hermite, Laguerre and Jacobi orthogonal polynomials (\cite{BGL_book}, \S~2.7).
These differential operators correspond to weighted Poincar\'e inequalities 
with weight $w(x) = x$ for the Gamma distribution $d\mu_1(x) \propto x^{\alpha-1} e^{- \alpha x} $ on $\R^+$, 
and weight $w(x) = 1 - x^2$ for the Beta distribution $d\mu_1(x) \propto (1 - x)^{\alpha-1} (1 + x)^{\beta - 1} $ on $[-1, 1]$. Notice that in \cite{Song_Weighted_Poincare}, $w$ is chosen such that the eigenfunction associated to $\lambda_1$ is a first-order polynomial. Except for the three cases mentioned above, the other eigenfunctions cannot be all polynomials. 


\section{Weight-free derivative global sensitivity measures} \label{sec:weight_free_DGSM}

The lower bounds of total indices obtained with generalized chaos expansions may involve weighted DGSM.
For instance, in PDO expansions, weights involve the eigenfunction derivatives (Equation (\ref{ineq:diffOpLB2})).
The presence of weight can be a drawback when the integral has to be estimated with a small sample size, 
as it can increase the variance of the Monte Carlo estimator. 
In this section, we show how to choose the two first orthonormal functions of GC expansions 
in order to obtain weight-free DGSM.
Interestingly, this is related to Fisher information and Cram\'er-Rao bounds.

\begin{prop}[Lower bounds with weight-free DGSM, for pdf vanishing at the boundaries] \label{prop:FisherLB1}
Assume that $\frac{\partial h (x) }{\partial x_1}$ is in $L^2(\mu)$,
and that the probability distributions $\mu_i$ are absolutely continuous on their support $(a_i, b_i)$
with $-\infty \leq a_i < b_i \leq +\infty$.
For each $i$, denote by $p_i$ the corresponding probability density function.
Assume that $p_i$ belongs to $H^1(\mu_i)$, do not vanish on $(a_i, b_i)$ 
but vanishes at the boundaries: $p_i(a_i) = p_i(b_i) = 0$. 
Finally, assume that $p_i'$ is not identically zero, and that $p_i'/p_i$ is in $L^2(\mu_i)$.
Define $Z_i(x_i) = (\ln p_i)'(x_i)$ and $I_i = \Var(Z_i(X_i))$.
Then, we have the inequality:
\begin{equation} \label{eq:DGSMlowerBound}
\Dt \geq  I_1^{-1}  c_1 ^2 + I_1^{-1} \sum_{j = 2}^d  I_j^{-1} c_{1,j}^2 
\end{equation} 
with 
\begin{eqnarray*}
c_1 &=& \int h(x) Z_1(x_1) \mu(dx) = - \int \frac{\partial h(x)}{\partial x_1}  \mu(dx) \\
c_{1, j} &=&  \int h(x) Z_1(x_1)Z_j(x_j) \mu(dx) 
         =   - \int \frac{\partial h (x) }{\partial x_1} Z_j(x_j) \mu(dx)
\end{eqnarray*}
Furthermore, if all the cross derivatives $\frac{\partial^2 h (x) }{\partial x_1 \partial x_j}$ 
are in $L^2(\mu)$, then
$$c_{1,j} =  \int \frac{\partial^2 h (x) }{\partial x_1 \partial x_j} \mu(dx)$$
The cases of equality correspond to functions $h$ of the form
\begin{equation} \label{propFisher:eqCase}
h(x) = \alpha_1 Z_1(x_1) + \sum_{j=2}^d \alpha_j Z_1(x_1)Z_j(x_j) + h(x_2, \dots, x_d).
\end{equation}
\end{prop}

\begin{proof}
For $i=1, \dots, d$, let $e_{i,1}(x_i) := I_i^{-1/2}Z_i(x_i)$. 
Then define 
$$\phi_1(x) = e_{1,1}(x_1), \quad \text{ and for } j=2, \dots, d: \quad \phi_j(x) = e_{1,1}(x_1) e_{j,1}(x_j).$$
By definition, the norm of each $e_{i,1}$ is equal to $1$. Furthermore, $Z_i$ is centered, since 
$$\Esp[Z_i] = \int_{a_i}^{b_i} p'_i(x_i)dx_i = \left[ p_i(x_i) \right]_{a_i}^{b_i} = 0.$$
This implies that $e_{i,1}$ is orthogonal to $e_{i,0} = 1$.
By Proposition~\ref{prop:HilbertBasis}, the $\phi_i$'s are then orthonormal functions of $\Hilb_1^\tot$.
The inequality is then given by Corollary~\ref{prop:LB_general_1var},
with first expressions of $c_1$ and $c_{1,j}$.
The other ones are obtained by integrating by part, 
using that the values at the boundaries of the $p_j$'s are zero.
\end{proof}

The proposition can be adapted when the probability density functions 
do not vanish at the boundaries of their support, by modifying the definition of the $Z_j$'s.
Notice that the expressions of $c_1$ and $c_{1,j}$ that involve derivatives 
then contain corrective terms, and are of limited practical interest.
For instance, denoting $\left[h \right]_{a_1}^{b_1} = h(b_1) - h(a_1)$
and $h_0 = \int h(x) \mu(dx)$, we have:
$$ c_1 = \left[ \left( \int h(x_1, x_{-1}) \mu_{-1}(d x_{-1}) - h_0 \right) p_1(x_1) \right]_{a_1}^{b_1} 
- \int \frac{\partial h(x)}{\partial x_1}  \mu(dx).$$
Nevertheless, the first expressions of $c_1$ and $c_{1,j}$ remain valid and, 
by analogy to Proposition~\ref{prop:FisherLB1},
have a close connection to derivative-based lower bounds.

\begin{prop}[[Lower bounds with weight-free DGSM, general case] \label{prop:FisherLB}
Assume that $\frac{\partial h (x) }{\partial x_1}$ is in $L^2(\mu)$,
and that the probability distributions $\mu_i$ are absolutely continuous on their support $(a_i, b_i)$
with $-\infty \leq a_i < b_i \leq +\infty$.
For each $i$, denote by $p_i$ the corresponding probability density function.
Assume that $p_i$ belongs to $H^1(\mu_i)$ and do not vanish on $(a_i, b_i)$.
Finally, assume that $p_i'$ is not identically zero, and that $p_i'/p_i$ is in $L^2(\mu_i)$.
Define $Z_i(x_i) = (\ln p_i)'(x_i) - \left[ p_i(x_i) \right]_{a_i}^{b_i}$ and $I_i = \Var(Z_i(X_i))$.
Then Inequality (\ref{eq:DGSMlowerBound}) holds with $c_1 = \int h(x) Z_1(x_1) \mu(dx)$ and 
$c_{1, j} =  \int h(x) Z_1(x_1)Z_j(x_j) \mu(dx)$. The equality case is the same as in Proposition~\ref{prop:FisherLB1},
and given by (\ref{propFisher:eqCase}).
\end{prop}

\newtheorem{rem1}{Remark}
\begin{rem1}
The expressions of $Z_i$ and $I_i$ in Proposition~\ref{prop:FisherLB} correspond 
respectively to the score and to the Fisher information at $\bm{\theta} = 0$ 
of a parametric family of probability distributions obtained by translation
$p_{i, \theta_i}(x_i) = p_i(x_i + \theta_i)$.
In this framework, the lower bound (\ref{eq:DGSMlowerBound}) corresponds to the Cram\'er-Rao lower bound.
\end{rem1}

\paragraph{Examples.}
First consider the case of normal distributions $\mu_i \sim \mathcal{N}(m_i, v_i)$ ($i=1, \dots, d$).
Applying Inequality (\ref{eq:DGSMlowerBound}) gives
\begin{equation} \label{eq:normalLB}
\Dt \geq v_1\left( \int \frac{\partial h(x)}{\partial x_1}  \mu(dx) \right)^2 
+ v_1 \sum_{j = 2}^d  v_j \left( \int 
\frac{\partial^2 h (x) }{\partial x_1 \partial x_j} \mu(dx) \right)^2.
\end{equation}
Here, the inequality is equivalent to Inequality~(\ref{ineq:diffOpLB2}) obtained with the Poincar\'e differential operator
of Section~\ref{sec:DO_expansions}, since $Z_i$ is a first-order polynomial, 
and thus equal to the first eigenvector of $L$ (Hermite polynomial).
The case of equality corresponds to functions of the form 
$$ h(x) = \alpha_1 (x_1 - m_1) + \sum_{j=2}^m \alpha_j (x_1 - m_1)(x_j - m_j) + g(x_2, \dots, x_d).$$

Other inequalities can be established for standard probability distributions.
Table~\ref{tab:derLowerBound} summarizes the results for some of them.
Notice that the equality case does not always correspond to polynomials
(see the form of $Z$).
Interestingly, an inequality is obtained for the Cauchy distribution, 
whereas the theory of Section~\ref{sec:DO_expansions} does not apply 
as this distribution does not admit a Poincar\'e constant.
On the other hand, some probability distributions for which Section~\ref{sec:DO_expansions} is applicable,
do not satisfy the assumptions of Proposition \ref{prop:FisherLB}, 
such as the uniform ($p'_i$ is identically zero) or the triangular distributions ($p_i'/p_i$ does not belong to $L^2(\mu_i)$).

\bigskip
\begin{table} 
\centering
\begin{tabular}{|c|c|c|c|c|} \hline
Dist. name & Support & $p$ & $Z$ & $I$ \\ \hline
Normal & $\R$ & $\frac{1}{s \sqrt{2\pi}} \exp \left( - \frac{1}{2} \frac{(x - m)^2}{s^2}\right)$ 
& $- (X - m)/s^2$ & $1 / s^2$ \\ \hline
Laplace & $\R$ & $\frac{1}{2s} \exp \left( \frac{\vert x - m \vert}{s} \right)$ 
& $- \sgn(X - m)/s$ & $1 / s^2$ \\ \hline
Cauchy & $\R$ & $\frac{1}{\pi} \frac{s}{(x-x_0)^2 + s^2}$
& $\frac{-2(x-x_0)}{(x - x_0)^2 + s^2}$ & $1/(2s^2)$\\ \hline
\end{tabular}
\caption{Useful quantities for derivative-based lower bounds. 
For readability, we have removed the subscript $j$ for $p, Z, I$.
The parameter $s$ is a scale parameter, and can be different from the standard deviation.}
\label{tab:derLowerBound}
\end{table}

\paragraph{Link to other works.}
Here, we briefly compare our lower bounds to those presented in the recent review \cite{Kucherenko_Iooss_book}.

For the uniform distribution on $[0, 1]$, 
we can obtain both a better upper bound 
and a description of the equality case.
For that, we apply Corollary~\ref{prop:LB_general_1var} to the orthonormal function obtained from $x_1^m$,
i.e. $\phi(x_1) = (x_1^m - m_1)/s_1$ with $m_1 = 1/(m+1)$ and $s_1^2 = \left(\frac{m}{m+1} \right)^2 \frac{1}{2m+1}$.
Then after some algebra and an integration by part, we obtain 
$$ \Dt \geq \frac{2m+1}{m^2} \left( \int (h(1, x_{-1}) - h(x)) dx - w_1^{(m+1)} \right)^2$$
where $w_1^{(m+1)} = \int \frac{\partial h(x)}{\partial x_1} x_1^{m+1} dx$. 
This improves on the lower bound found in \cite{Kucherenko_Iooss_book}, Theorem 2, which has the same form, 
but with the smaller multiplicative constant $\frac{2m+1}{(m+1)^2}$.
Furthermore, the lower bound above is attained when $h$ has the form $h(x) = \alpha_1 x_1^m + g(x_2, \dots, x_d)$.
However, notice that these two lower bounds are only a lower bound for $D_1 \leq \Dt$, 
and can be improved by considering additional orthonormal functions belonging to $\Hilb_1^\tot \setminus \Hilb_1$. 

For normal distributions, Inequality (\ref{eq:normalLB}) 
improves the lower bound given by \cite{Kucherenko_Song_2016}, i.e.
$$ \Dt \geq v_1 \left( \int \frac{\partial h(x)}{\partial x_1}  \mu(dx) \right)^2.$$
Here also, this latter lower bound is only a lower bound of $D_1 \leq \Dt$ 
since it corresponds to the case in Corollary~\ref{prop:LB_general_1var} 
where the $\phi_j$'s (here $\phi_1(x) = Z_1(x_1)/I_1^{1/2}$) only depend on $x_1$.


\section{Examples on analytical functions} \label{sec:examples}

This section briefly illustrates PDO expansions for the uniform distribution
on benchmark functions from sensitivity analysis.
We assess the accuracy of the lower bounds of total indices, when only the two first eigenvalues are used.

\subsection{A polynomial function with interaction} 

\begin{example} Let us consider $g(x_1, x_2) = x_1 + a x_1 x_2$, 
and let $\mu$ be the uniform distribution on $[-1/2, 1/2]^2$.
The inequalities obtained by truncating the PDO expansions to the first eigenvalue are:
\begin{eqnarray*}
D_1 = \frac{1}{12} &\approx& 0.0833  \geq  0.0821 \approx \frac{8}{\pi^4} \\ 
\label{eq:polynomLBmain}
D_1^\tot = \frac{1}{12} + \frac{a^2}{144} &\approx & 0.0833 + 0.0069 \, a^2   \\
& \geq &  0.0821 + 0.0067 \, a^2 \approx \frac{8}{\pi^4} + \frac{64}{\pi^8} \, a^2
\label{eq:polynomLBtotal}
\end{eqnarray*}
\end{example}

We can see that for a polynomial function of degree $1$ with respect to $x_1$, 
the lower bound obtained by restricting the PDO expansion to the first eigenvalue is very accurate.
Hence, we do not loose a lot of information by ignoring that the function is a polynomial.
This is an ideal situation for polynomial chaos.\\ 

Let us give some computing details on the previous inequalities.
It is easy to check that the two terms $x_1$, $a x_1 x_2$ 
correspond to the main effect and second order interaction respectively. 
The partial variances are given by $D_1 = 1/12$ and $D_{1,2} = a^2/144$. 
Hence, $D_1^\tot = 1/12 + a^2/144$.
Restricting the PDO expansion to the first term, a lower bound is given by Inequality~(\ref{ineq:DOFourierDer}):
$$ D_1^\tot \geq  \frac{2}{\pi^2} \left( \langle \frac{\partial g}{dx_1}, \cos(\pi x_1) \rangle^2 
+ 2 \langle \frac{\partial g}{dx_1}, \cos(\pi x_1) \sin(\pi x_2) \rangle^2  \right).$$
The two terms of the lower bound above correspond to a lower bound of $D_1$ and $D_{1,2}$ respectively. 
A direct computation gives:
\begin{eqnarray*}
\textrm{LB}_1 &:=&  \frac{2}{\pi^2} \langle \frac{\partial g}{dx_1}, \cos(\pi x_1) \rangle^2 
= \frac{2}{\pi^2} \left( \frac{2}{\pi} \right)^2 = \frac{8}{\pi^4} \\
\textrm{LB}_{1,2} &:=&  \frac{2}{\pi^2} 2 \langle \frac{\partial g}{dx_1}, \cos(\pi x_1) \sin(\pi x_2) \rangle^2 
= \frac{2}{\pi^2} . 2 . \left( \frac{4a}{\pi^3} \right)^2 = \frac{64a^2}{\pi^8} 
\end{eqnarray*}
The result follows.

\subsection{A separable function}
\begin{example} 
Consider the g-Sobol' function on $[-1/2, 1/2]$ defined by
$$ g(x) = \prod_{i = 1}^d (1 + h_i(x_i)) $$
with $h_i(x_i) = (4 \vert x_i \vert - 1)/ (1+a_i)$ ($i=1, \dots, d$),
and let $\mu$ be the uniform distribution on $[-1/2, 1/2]^d$.
The inequalities obtained by truncating the PDO expansions to the first two eigenvalues are:
\begin{eqnarray}
D_i = \frac{1}{3} \frac{1}{(1+a_i)^2} & \geq & \frac{32}{\pi^4} \frac{1}{(1+a_i)^2} := \textrm{LB}_i \\ 
\label{eq:gSobolLBmain}
D_i^\tot = D_i \prod_{j \neq i}^d (1 + D_j) & \geq & \textrm{LB}_i . \sum_{j \neq i}^d \textrm{LB}_j  
\label{eq:gSobolLBtotal}
\end{eqnarray}
\end{example}

Notice that $32/\pi^4 \approx 0.328$ is very close to $1/3$. 
Hence, the lower bound for $D_i$ is very accurate.
Obviously, a very sharp inequality $D_i^\tot \geq \textrm{LB}_i \prod_{i \neq j}^d (1 + \textrm{LB}_j)$ 
could have been deduced, but this is unrealistic in practice, since the separable form of the function is unknown.
The lower bound~(\ref{eq:gSobolLBtotal}) for $D_i^\tot$ is actually a very good approximation of
the variance explained by second-order interactions involving $x_i$, equal to $D_i . \sum_{j \neq i}^d D_j$.
Hence, Inequality (\ref{eq:gSobolLBtotal}) will be less fine in presence of higher order interactions, 
(tuned by the values of the $a_j$'s).
Then, more than two eigenvalues in PDO expansions must be considered.\\

Let us give some computing details on the previous inequalities.
Without loss of generality, we write the proof for $i=1$.
Let us first recall the computation of Sobol' indices for the g-Sobol' function.
As all the $h_i$ are centered, the Sobol'-Hoeffding decomposition is given by $g_I(x_I) = \prod_{i \in I} h_i(x_i)$.
In particular $D_1 = \int h_1^2 d\mu_1 = \frac{1}{3} \frac{1}{(1+a_1)^2}$.
Furthemore, the variance of a second order interaction is, for $ i \neq 1$:
\begin{equation} \label{eq:gSobolAll2ndOrder}
D_{1,i} = E( h_1(x_1)^2 h_i(x_i)^2) = D_1 D_i,
\end{equation}
and variance explained by second-order interactions containing $x_1$ is equal to 
$$\sum_{i=2}^d D_{1,i} = D_1 \sum_{i=2}^d D_i.$$
Finally the total effect is the variance of $\sum_{I \supseteq \{1 \}} \prod_{i \in I} h_i$, equal to
$$ D_1^\tot = \sum_{I \supseteq \{1 \}} \prod_{i \in I} D_i = D_1 \prod_{i=2}^d (1 + D_i).$$

Let us now consider lower bounds. 
To obtain accurate lower bounds, we need to consider the first two non-zero eigenvalues. 
Indeed, the first non-zero eigenvector is even and all the dot products are 0.
By using Equation~(\ref{eq:diffOpLB2}) and the results about uniform distributions 
presented in Section~\ref{sec:DO_expansions}, we obtain:
\begin{equation} \label{eq:spectral2ndOrder}
D_1^\tot(g) \geq 
\frac{1}{\lambda_2^2} \left( \langle \frac{\partial g}{dx_1}, e'_{1,2} \rangle^2 
+ \sum_{i=2}^d  \langle \frac{\partial g}{dx_1}, e'_{1,2} e_{i, 2}\rangle^2 \right),
\end{equation}
with $e_{i,2} = \sqrt{2} \cos(2\pi x_i)$ (we omit the '-' sign) and $\lambda_2 = 4 \pi^2$.
We could have also used (\ref{eq:diffOpLB1}), 
but using derivatives simplifies the computations here.

The first term gives a lower bound for $D_1$. We have:
$$ \frac{\partial g}{dx_1}(x) = \frac{4}{1+a_1} \sgn(x_1) \prod_{i \geq 2} (1 + h_i(x_i)).$$
Due to the tensor form of the g-Sobol' function partial derivative, 
the dot product is expressed as a product of one-dimensional dot-products.
Furthermore, as all the $h_i's$ are centered, the dot-products in dimensions $2, \dots, d$ are equal to 1. 
Finally,
\begin{eqnarray*}
\langle \frac{\partial g}{dx_1}, e'_{1,2} \rangle = \langle h'_1, e'_{1,2} \rangle_1 
&=& \frac{4}{1+a_1} \int_{-1/2}^{1/2} \sgn(x_1)e'_{1,2}(x_1) dx_1 \\
&=& \frac{4}{1+a_1} \sqrt{2}. 2 \int_0^{1/2} 2 \pi \sin(2 \pi x_1) dx_1 = \frac{16 \sqrt{2}}{1+a_1}.\\
\end{eqnarray*}
This gives the announced lower bound for the main effect (Equation~(\ref{eq:gSobolLBmain})):
\begin{equation} \label{eq:gSobolLB}
\textrm{LB}_1 = \frac{1}{\lambda_2^2} \langle \frac{\partial g}{dx_1}, e'_{1,2} \rangle^2 =
\frac{1}{(4\pi^2)^2} \left(\frac{16 \sqrt{2}}{1+a_1} \right)^2 = \frac{32}{\pi^4} \frac{1}{(1+a_1)^2}.
\end{equation}

Now, let us compute the second term in (\ref{eq:spectral2ndOrder}). 
Notice that it is a lower bound for the variance explained by second-order interactions involving $x_1$,
as computed in (\ref{eq:gSobolAll2ndOrder}). As above, exploiting the tensor form, we have:
$$ \langle \frac{\partial g}{dx_1}, e'_{1,2} e_{i,2} \rangle 
= \langle h'_1, e'_{1,2} \rangle_1 
. \langle (1+h_i), e_{i,2} \rangle_i.$$
The first term has already been computed above. For the second one, we use the property of eigenvectors (\ref{eq:intPart}):
$$ \langle (1+h_i), e_{i,2} \rangle_i = \frac{1}{\lambda_2} \langle h'_i, e'_{i,2} \rangle_i$$
and we recognize the quantity computed above where we replace $1$ by $i$, equal to 
${\displaystyle \frac{1}{\lambda_2}\frac{16 \sqrt{2}}{1+a_i}} = \sqrt{\textrm{LB}_i}$.
Finally, plugging this result in (\ref{eq:spectral2ndOrder}) together with (\ref{eq:gSobolLB})
gives the announced lower bound~(\ref{eq:gSobolLBtotal}).


\section{Applications} \label{sec:applications}

In this section, two numerical models representing real physical phenomena are used in order to illustrate the usefulness of the
lower bounds of total Sobol' indices provided by PDO expansions. 
More precisely, we restrict ourselves to the simplest lower bound provided by considering only the first eigenfunctions in all dimensions, 
given by the two equivalent Equations~(\ref{ineq:diffOpLB1}) and (\ref{ineq:diffOpLB2}).
The first equation gives a derivative-free lower bound of the total index, here called \emph{PDO lower bound.} 
The second one gives a derivative-based version, here called \emph{PDO-der lower bound}. 

Whereas the PDO and PDO-der lower bounds are theoretically equal, their estimated values will differ.
Estimations of integrals and square products have been performed via crude Monte Carlo samples.
We have centered the function $f$.
It does not change the value of sensitivity indices but reduces the estimation error. 
The use of Monte Carlo samples allows to provide confidence intervals on the estimates by the way of a bootstrap resampling technique.
Boxplots will be used to graphically represent these estimation uncertainties.
Finally, the computation of eigenvalues, eigenfunctions and eigenfunction derivatives 
has been done with the numerical method presented in \cite{Roustant_Barthe_Iooss_2017}.

\subsection{A simplified flood model}

Our first model simulates flooding events by comparing the height of a river to the height of a dyke.
It involves the characteristics of the river stretch, as already studied in \cite{Lamboni_et_al_2013,Roustant_Barthe_Iooss_2017}.
The model has $8$ input random variables (r.v.), each one follows a specific probability distribution (truncated Gumbel, truncated normal, triangular or uniform). 
When the height of a river is over the height of the dyke, flooding occurs. 
The model output is the cost (in million euros) of the damage on the dyke which writes:
\begin{equation}
Y = \hbox{1\kern-.24em\hbox{I}}_{S>0} +  \left[0.2 + 0.8\left( 1-\exp^{-\frac{1000}{S^4}}\right) \right] \hbox{1\kern-.24em\hbox{I}}_{S \leq 0} + \frac{1}{20}\left(H_d \hbox{1\kern-.24em\hbox{I}}_{H_d>8} + 8 \hbox{1\kern-.24em\hbox{I}}_{H_d \leq 8} \right) \,,
\end{equation}
where $\hbox{1\kern-.24em\hbox{I}}_{A}(x)$ is the indicator function which is equal to 1 for $x \in A$ and 0 otherwise, $H_d$ is the height of the dyke (uniform r.v.) and $S$ is the maximal annual overflow (in meters) based on a crude simplification of the 1D hydro-dynamical equations of Saint-Venant under the assumptions of uniform and constant flowrate and large rectangular section.
$S$ is calculated as
\begin{equation}
S = \left(\frac{Q}{BK_s \sqrt{\frac{Z_m-Z_v}{L} }} \right)^{0.6} + Z_v - H_d - C_b \,,
\end{equation}
with $Q$ the maximal annual flowrate (truncated Gumbel r.v.), $K_s$ the Strickler coefficient (truncated Gaussian r.v.), $Z_m$ and $Z_v$ the upstream and downstream riverbed levels (triangular r.v.), $L$ and $B$ the length and width of the water section (triangular r.v.) and $C_b$ the bank level (triangular r.v.).
For this model, first-order and total Sobol' indices have been estimated in \cite{Lamboni_et_al_2013} with high precision (large sample size) via a Monte-Carlo based algorithms.

Fig. \ref{fig:anovaflood} shows the PDO lower bounds. 
By looking at the values of first-order and total Sobol' indices (horizontal straight lines), we notice that rather large interaction effects are present between four inputs of the model ($Q$, $K_s$, $Z_v$ and $H_d$). 
First, the bounds estimated with the sample size $n=100$ have large uncertainties.
It shows that this sample size is too small for this complex model (it includes non-linear and interaction effects).
Secondly, concerning the estimation of the bounds, the convergence is reached, with very small uncertainties on the estimates from $n=10\,000.$
From this sample size, we can visually check (e.g. looking at the third quartile) 
that estimated lower bounds are smaller than the corresponding true Sobol' indices.
Moreover, for smaller sample sizes as $n=1\,000$, results for all the inputs show sufficient accuracies (easy discrimination between the bounds).
Finally, except for $K_s$ and $H_d$, the bounds are informative because:
\begin{itemize}
\item The PDO bounds are very close to the theoretical values of total Sobol' indices, which is remarkable as only the first eigenvalue was used. 
\item The PDO lower bounds for total indices are larger than their respective first-order Sobol' indices.
\end{itemize}

\begin{figure}[!ht]
\begin{center}
\includegraphics[width=\textwidth]{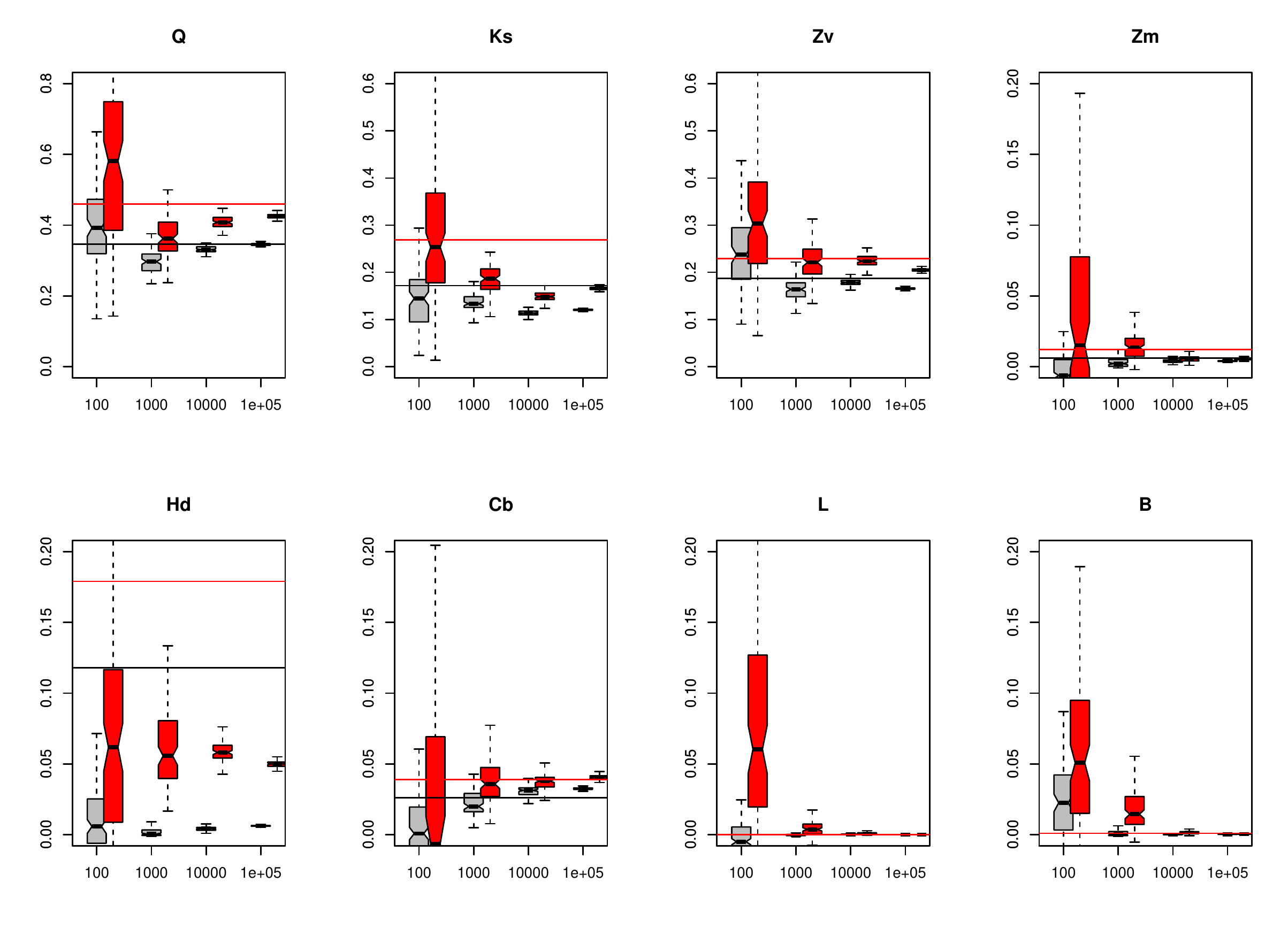}
\end{center}

\vspace{-0.8cm}
\caption{PDO bounds for the $8$ inputs of the flood model application for four different sample sizes $n$ ($10^2$, $10^3$, $10^4$ and $10^5$). Red (resp. gray) boxplots are lower bounds of total (resp. first-order) Sobol' indices. 
Horizontal lines indicate the `true' values of the Sobol' indices.}
\label{fig:anovaflood}
\end{figure}

Fig. \ref{fig:dgsmflood} shows that the PDO-der lower bounds give significantly better results than the PDO bounds, especially for small sample sizes. 
In particular, when the Sobol' indices are close to zero, the bounds perfectly match their respective Sobol' indices from $n=100$.
This result clearly favors the use of derivative-based lower bounds for the screening step when model derivatives can be computed.

\begin{figure}[!ht]
\begin{center}
\includegraphics[width=\textwidth]{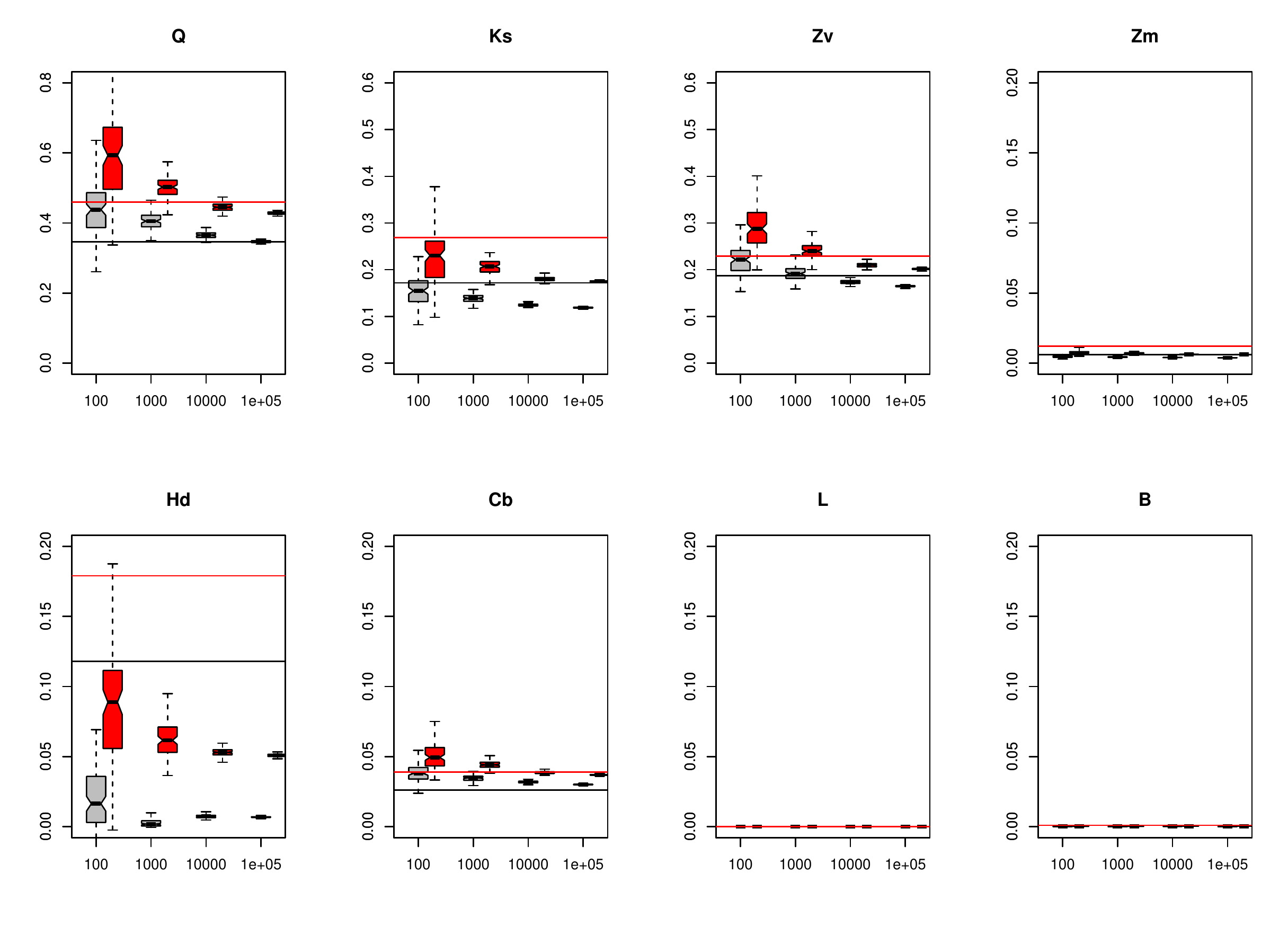}
\end{center}

\vspace{-0.8cm}
\caption{PDO-der bounds for the $8$ inputs of the flood model application. 
The legend details are the same as Figure~\ref{fig:anovaflood}.}
\label{fig:dgsmflood}
\end{figure}

\subsection{An aquatic prey-predator chain}

This application is related to the modeling of an aquatic ecosystem called MELODY (MESocosm structure and functioning for representing LOtic DYnamic ecosystems).
This model simulates the functioning of aquatic mesocosms as well as the impact of toxic substances on the dynamics of their populations. 
Inside this model, the Periphyton-Grazers sub-model is representative of processes involved in dynamics of primary producers and primary consumers, i.e. photosynthesis, excretion, respiration, egestion, mortality, sloughing and predation \cite{circif12}. 
It contains a total number of $d = 20$ uncertain input variables.
In order to conduct sensitivity analysis, \cite{circif12} has defined that each of these input variables are random following a uniform distribution law, defined by their minimal and maximal values.

The PDO-der upper bound of total Sobol' indices \cite{sobkuc09} was then applied in \cite{ioopop12} on one model output (the periphyton biomass) at only one reference time, day $60$ of simulations, which corresponds to the period of maximum periphyton biomass and a growth phase for grazers, according to experimental data.
A design of experiments of size $n=100$ was then provided, and simulated with MELODY.
A model output vector of size $100$ is obtained, as well as the derivatives of the output with respect to each input at each point of the design (matrix of size $100 \times 20$).
In this section, we analyze the same data that has been studied in \cite{ioopop12}.

Fig. \ref{fig:anovagrazer} shows the PDO lower bounds, as well as the first-order Sobol' indices estimates (via the local polynomials sample based technique \cite{davwah09}).
Good results are obtained on the first-order lower bounds which have reduced estimation uncertainties and are always smaller than 
the estimated first-order Sobol' indices. 
Less accurate estimates are obtained for the lower bounds of total indices. 
They remain informative because they are clearly larger than the first-order Sobol' indices.
This last result proves that large interactions between inputs dominate in this prey-predator model, which confirms the first analysis of \cite{ioopop12} (the sum of all the first-order Sobol' indices is much smaller than one).
The new results of Fig. \ref{fig:anovagrazer} prove the strong influence of some inputs which have large total lower bounds.
For example, $5$ inputs have total lower bound median values larger than $20\%$: Maximum photosynthesis rate (n$^\circ$1), Maximum consumption rate (n$^\circ$2), Rate of change per  $10^{\circ}$C (n$^\circ$9), Grazers preference for periphyton (n$^\circ$11) and Intrinsic mortality rate (n$^\circ$16).
This result cannot be found from the first-order Sobol' indices which are rather small (except for the Maximum photosynthesis rate).

\begin{figure}[!ht]
\begin{center}
\includegraphics[width=0.95\textwidth]{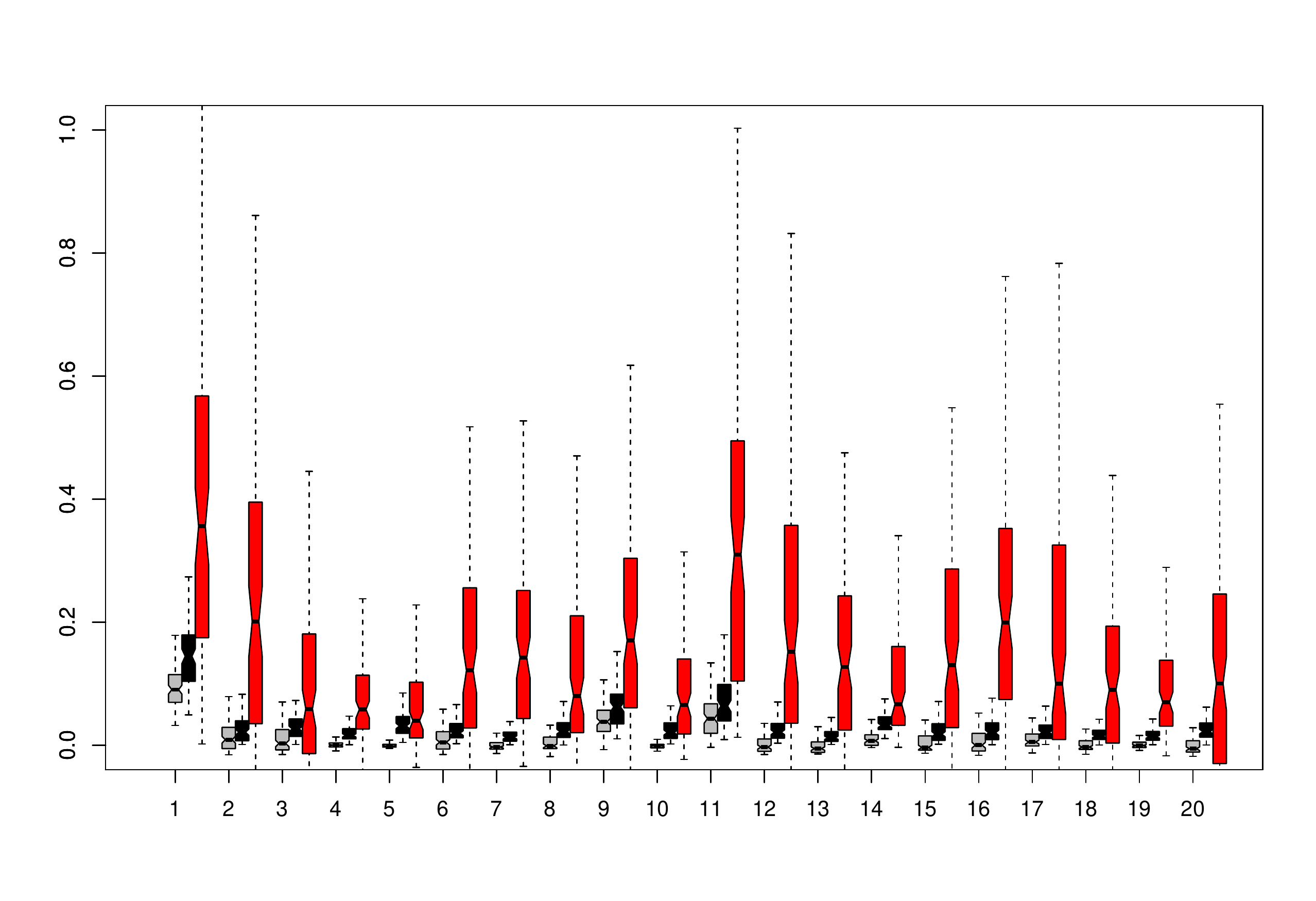}
\end{center}

\vspace{-0.8cm}
\caption{PDO bounds for the $20$ inputs of the prey-predator model. Gray, black and red boxplots are respectively the lower bounds of the first-order Sobol' indices, the estimates of the first-order Sobol' indices and the lower bounds of the total Sobol' indices.}\label{fig:anovagrazer}
\end{figure}

Fig. \ref{fig:dgsmgrazer} shows the PDO-der lower bounds, as well as the PDO-der upper bounds of the total Sobol' indices (see \cite{ioopop12}) whose confidence intervals are also obtained by bootstrap. 
In this figure lower and upper bounds of total Sobol' indices have been truncated to one in order to only consider realistic values.
Indeed, values larger than one are theoretically impossible but can sometimes be found due to numerical estimation errors.
First, some partial checks can be done by looking at the median of the estimated values, 
e.g. by observing that the lower bounds are smaller than the upper bounds for each input.
Second, several PDO-der lower bounds estimates are much less accurate than the (derivative-free) PDO lower bounds, 
especially when their values are large, for example the inputs n$^\circ$1 and n$^\circ$11.
Even in this case of large values, informative results can be deduced by taking their median values: 
the total Sobol' indices of the input n$^\circ$1 (resp. n$^\circ$11) approximately lie in $[0.85,1]$ (resp. $[0.4,1]$).
From these total lower bounds, a coarse importance hierarchy can then be proposed between the most influential inputs.

\begin{figure}[!ht]
\begin{center}
\includegraphics[width=0.95\textwidth]{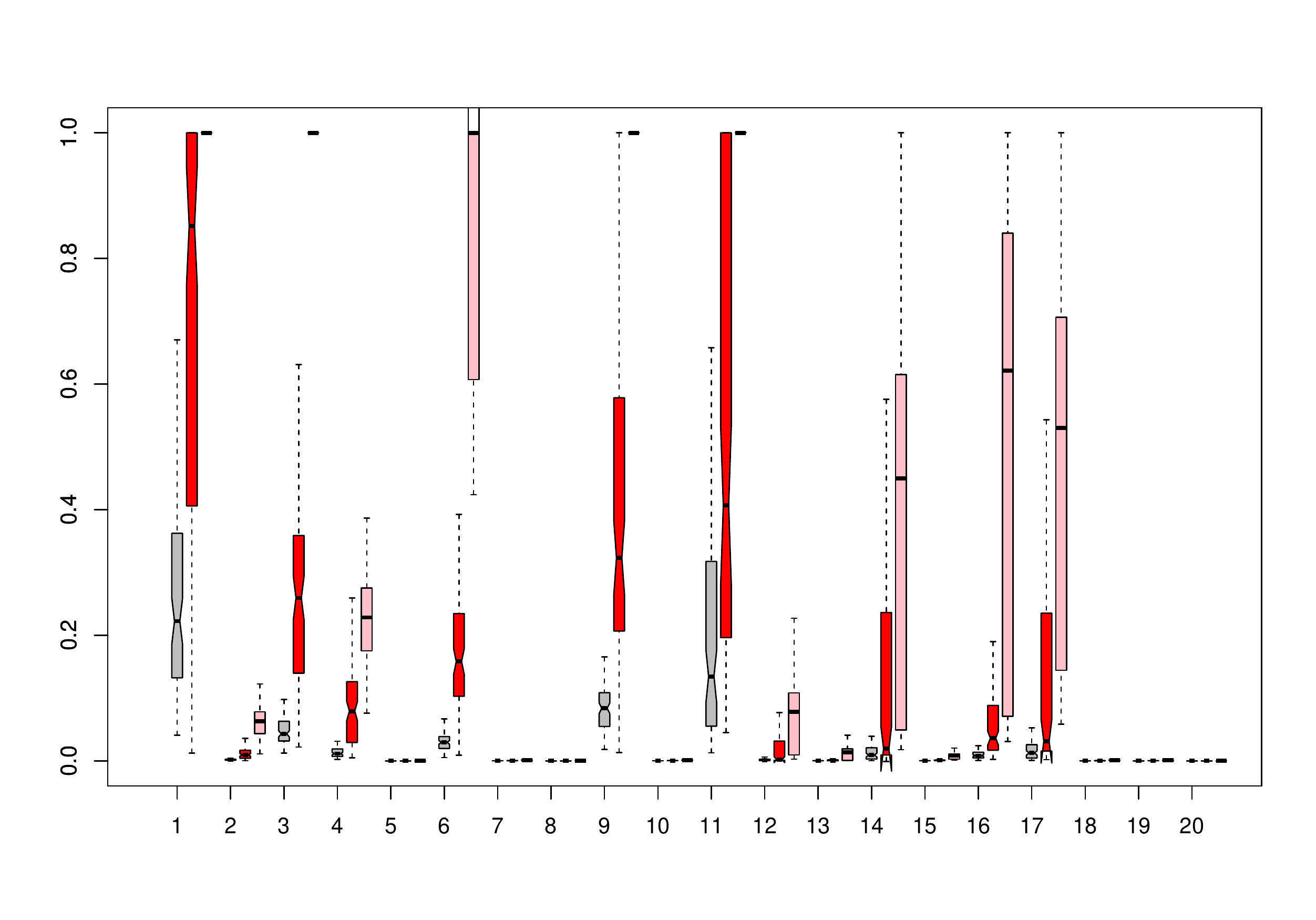}
\end{center}

\vspace{-0.8cm}
\caption{PDO-der bounds for the $20$ inputs of the prey-predator model. As in Figure~\ref{fig:anovagrazer}, 
gray and red boxplots are respectively lower bounds of the first-order and total Sobol' indices.
The additional pink boxplots correspond to the upper bound of the total Sobol' indices.}
\label{fig:dgsmgrazer}
\end{figure}

Finally, we observe the excellent results for non influential inputs which have all their PDO-der lower and upper bounds close to zero (inputs 2, 5, 7, 8, 10, 13, 15, 18, 19, 20).
This is not the case with the PDO lower bounds (see Fig. \ref{fig:anovagrazer}) which are more difficult to exploit.
A convenient usage would be to estimate both derivative-free and derivative-based lower bounds, and to keep the smallest value.
Indeed, the PDO bound is more accurate when the Sobol' index is much larger than zero, 
whereas the PDO-der bound is much smaller when the Sobol' index is close to zero.

\subsection{Conclusion on the applications}

On the two previous applications, we have tested the simplest PDO and PDO-der lower bounds, 
obtained by keeping only the first eigenvalue in all dimensions, for real-world models involving non-linear and interaction effects.
Several conclusions can be made:
\begin{itemize}
\item Lower bounds can be easily computed for any probability distribution of the inputs;
\item The estimation error can be large for small sample sizes. Estimating some boostrap confidence intervals 
is essential to evaluate the quality of the estimates;
\item The lower bounds of the total Sobol' indices are most of the times informative, 
i.e. larger than the (estimated) first order Sobol' indices;
\item Using derivatives (then DGSM) is sometimes preferable to obtain lower bounds, especially for the screening step (identification of non influential inputs with negligible total Sobol' indices).
With DGSM, excellent results are obtained for screening, even for small sample size cases. 
\end{itemize}
 

\section{Further works} \label{sec:conclusion}
In this paper, we revisit the so-called chaos expansion method for the evaluation of Sobol' indices. We summarize in a compact way the role played by the functional basis and the associated projection operators for evaluating by below these indices through a truncated Parseval formula. Generalized chaos basis built on the Poincar\'e diferential operator associated to the input distribution leads to very interesting new lower bounds for the total Sobol' index in terms of DGSM. This bound appears to be sharp both on toy and real life models, allowing a fast screening of the model input based on the energy of the function derivatives. This opens some challenging problems in mathematical statistics. First, the bounds obtained by the brute force truncation method could certainly been merely improved considering accurate model selection methods as adaptive thresholding or $l^1$ regularization. Second, the statistical estimation of the lower bound is a non linear semi-parametric problem. By non linear, we mean that the quantity to be estimated depends in a non linear way (here quadratic), of the infinite dimensional parameter (the function of interest). The estimation of a quadratic functional have been addressed in \cite{laurent1996efficient,
laurent2000adaptive,gine2008simple,da2013efficient}. It involves $U$-statistics theory, and offers an excellent source of inspiration for further works in mathematical statistics having concrete computational applications. For example, the unbiased estimation of such quantity for small sample appears to be an interesting challenging issue. As ending remark, notice that the use of PDO also opens challenging questions concerning the construction of such operators (and eigenbasis). First, one may be interested to build a PDO that provides a lower bound involving weighted DGSM. Secondly, one may wish to consider the case of heavy tail input distributions (as the Cauchy one for example).

\section*{Software and acknowledgement}
The implementations are partially based on the R package \texttt{sensitivity} \cite{sensitivityPackage}. 
The whole code should be included in a future version of that package.

Part of this research was conducted within the frame of the Chair in Applied Mathematics OQUAIDO, gathering partners in technological research (BRGM, CEA, IFPEN, IRSN, Safran, Storengy) and academia (CNRS, Ecole Centrale de Lyon, Mines Saint-Etienne, University of Grenoble, University of Nice, University of Toulouse) around advanced methods for Computer Experiments. The authors thank the participants for fruitful discussions. 
In particular we are grateful to A. Joulin for the insightful idea of using the Poincar\'e differential operator for computing lower bounds. Support from the ANR-3IA Artificial and Natural Intelligence Toulouse Institute is gratefully acknowledged.

\bibliographystyle{abbrv}
\bibliography{Sobol_lower_bound}

\end{document}